\documentclass[11pt,a4paper]{article}

\usepackage{helvet}

\usepackage[utf8]{inputenc}
\usepackage[colorlinks=true, urlcolor=blue]{hyperref}

\usepackage{amsfonts}
\usepackage{amssymb,amsmath,mathrsfs}
\usepackage{amsthm}
\usepackage{bbm}
\usepackage{mathtools}
\usepackage[margin=1.1in, includehead, includefoot]{geometry}
\usepackage{datetime}
\usepackage{dsfont}
\usepackage{verbatim}
\usepackage{enumerate}
\usepackage{MnSymbol}
\usepackage[titletoc, title]{appendix}
\usepackage{tikz}
\usepackage{caption}
\usepackage{float}
\usepackage{arydshln}
\usepackage{fancyhdr}

\usepackage[en-US]{datetime2}
\usepackage[nottoc, notlot, notlof]{tocbibind}

\newtheorem{theorem}{Theorem}[section]
\newtheorem*{theorem*}{Theorem}
\newtheorem{proposition}[theorem]{Proposition}
\newtheorem{lemma}[theorem]{Lemma}
\newtheorem{corollary}[theorem]{Corollary}
\newtheorem{example}[theorem]{Example}
\newtheorem{condition}{Condition}

\newtheorem{definition}[theorem]{Definition}

\newcommand{\wrt}[1]{\, \mathrm{d} #1}

\newcommand{\pd}[2]{\frac{\partial #1}{\partial #2}}

\newcommand{\fqquad}{\qquad \qquad \qquad \qquad}
\newcommand{\abs}[1]{\left\lvert#1\right\rvert}
\newcommand{\norm}[1]{\left\lVert#1\right\rVert}

\author{Nengli Lim}
\begin{document}
\title{\huge Young-Stieltjes integrals with respect to Volterra covariance functions}
\maketitle

\begin{abstract}
Complementary regularity between the integrand and integrator is a well known condition for the integral $\int_0^T f(r) \wrt{g(r)}$ to exist in the Riemann-Stieltjes sense. This condition also applies to the multi-dimensional case, in particular the 2D integral \\ $\int_{[0, T]^2} f(s,t) \wrt{g(s,t)}$. In the paper, we give a new condition for the existence of the integral under the assumption that the integrator $g$ is a Volterra covariance function. We introduce the notion of strong H\"{o}lder bi-continuity, and show that if the integrand possess this property, the assumption on complementary regularity can be relaxed for the Riemann-Stieltjes sums of the integral to converge.
\end{abstract}

\section{Introduction}
The theory of the Young-Stieltjes integral is one of the main foundations of rough paths theory \cite{fh2014, lyons98, lq2003, lcl2006}, which gives new estimates and insight into stochastic integrals by allowing them to be examined in a path-wise, or deterministic, manner in addition to the usual probabilistic approach. In the one-dimensional case, the key contribution of Young-Stieltjes integration is that the integrator need not be of bounded variation; rather, the integrator and integrand need only have complementary regularity, i.e. they only need to be of finite $p$ and $q$ variation respectively where $\frac{1}{p} + \frac{1}{q} > 1$. \par 

This notion of complementary regularity for integration has also been extended to the multi-dimensional case \cite{towghi2002}. For a function of two variables, 2D $p$-variation over a region, say $[0, T]^2$, is defined as
\begin{align} 
\left\| f\right\|_{p-var;[0,T]^{2}}:=\sup_{\pi }\left( \sum_{i,j} \left|
f \begin{pmatrix}
u_{i},u_{i+1} \\ 
v_{j},v_{j+1}
\end{pmatrix}
\right|^p \right) ^{\frac{1}{p}},
\end{align}
where the supremum runs over all partitions $\pi := \left\{ \left( u_i, v_j	\right) \right\}$ of $[0, T]^2$. Theorem 1.2(a) in \cite{towghi2002} shows that the Stieltjes sums of $\int_{[0, T]^2} f \wrt{g}$ converge if the integrand and integrator have complementary 2D variation. In particular, this result has been used in \cite{fv2010a} to construct the geometric rough path over Gaussian processes with low regularity, and in the recent paper \cite{cl18}, which gives a link between the Cameron-Martin norm and the 2D Young integral, and allows for the generalization of the classical It\^{o}-Stratonovich formula beyond standard Brownian motion. \par

In this paper, we will show that given that the integrator is a covariance function, and under a certain natural condition on its structure, one can bypass the requirement of complementary regularity for the existence of $\int_{[0, T]^2} f \wrt{g}$ in the Young-Stieltjes sense.  More precisely, we can weaken this requirement and consider integrands with finite 2D $p$-variation where $p \geq 3$, even if the integrator only has finite 2D $q$-variation with $q \geq \frac{3}{2}$. This case arises in situations where the functions have low regularity, e.g. when considering integrals w.r.t. the covariance function of fractional Brownian motion with Hurst parameter $H \leq \frac{1}{3}$ \cite{fv2010a}. \par 

Before we proceed, we will review other results in the literature to motivate the need for a new approach. To simplify things and be concrete, we will illustrate their inadequacy by considering integrands of the form 
\begin{align} \label{productEg}
f(s, t) = f_1(s) f_2(t),
\end{align}
where we assume, for reasons of symmetry, that $f_1$ and $f_2$ are both continuous with finite $p$-variation for $p \geq 3$ but infinite $p$-variation if $p < 3$. \par

One of the earliest results for multi-dimensional integration can be found in Young's original paper \cite{young1938}; see also \cite{ohashi} for a recent application to Brownian motion local time. Here, the result is formulated in terms of bivariations rather than complementary regularity. We say that $f: [0, T]^2 \rightarrow \mathbb{R}$ has finite $(p, q)$-bivariation if 
\begin{align*}
&\sup_{t_1, t_2 \in [0, T]} \norm{f(\cdot, t_2) - f(\cdot, t_1)}_{p-var; [0, T]} < \infty \quad \mathrm{and} \\
&\sup_{s_1, s_2 \in [0, T]} \norm{f(s_2, \cdot) - f(s_1, \cdot)}_{q-var; [0, T]} < \infty.
\end{align*}
Moreover, if we also assume that the integrator $g$ satisfies 
\begin{align} \label{biVarGCond}
\abs{g\begin{pmatrix}
s_i & s_{i+1} \\
t_j & t_{j+1}
\end{pmatrix}}
\leq C \abs{s_{i+1} - s_i}^{\frac{1}{\tilde{p}}} \abs{t_{j+1} - t_j}^{\frac{1}{\tilde{q}}},
\end{align}
(the definition of the rectangular increment is given in \eqref{rectInc}) and if there exists $\alpha \in (0, 1)$ such that
\begin{align} \label{bivarCond}
\frac{\alpha}{p} + \frac{1}{\tilde{p}} > 1 \quad \mathrm{and} \quad \frac{(1 - \alpha)}{q} + \frac{1}{\tilde{q}} > 1,
\end{align}
then the integral $\int_{[0, T]^2} f \wrt{g}$ exists in the Young-Stieltjes sense; see e.g. Corollary 1.1 in \cite{ohashi}. \par 

However, we see that if the integrand is of product form \eqref{productEg}, then it is at best of $(p, p)$-bivariation with $p \geq 3$, and we can also take $\alpha = \frac{1}{2}$ by symmetry. \eqref{bivarCond} then dictates that the integrator must satisfy \eqref{biVarGCond} with $\tilde{p} = \tilde{q} < \frac{6}{5}$ for the existence of the integral. \par 

Returning to \cite{towghi2002}, Theorem 1.2(b) in the paper gives a slight generalization on the condition on complementary regularity by also allowing one to consider the variation on each variable separately. To do so, we have the notion of mixed variation, where we say that $f \in V^{(p, q)-var}([0, T]^2)$ if 
\begin{align*}
\sup_{\pi} \left( \sum_i \left( \sum_j \left| f 
\begin{pmatrix}
s_i & s_{i+1} \\ 
t_j & t_{j+1}
\end{pmatrix}
\right|^p \right)^{\frac{q}{p}} \right)^{\frac{1}{q}} < \infty,
\end{align*}
and the 2D integral exists if the integrator is in $V^{(\tilde{p}, \tilde{q})-var}([0, T]^2)$ with $\frac{1}{p} + \frac{1}{\tilde{p}} > 1$ and $\frac{1}{q} + \frac{1}{\tilde{q}} > 1$. Although this is an improvement over the bivariation condition in \eqref{bivarCond}, we see that in the symmetric case \eqref{productEg}, $p$ must equal $q$ and we are back to the case where $f \in V^{p-var}([0, T]^2)$ with $p \geq 3$, which means that the integrator is required to be of finite 2D $q$-variation with $q < \frac{3}{2}$. \par 

In this paper, we will improve the existing results such that $q$ can be greater than $\frac{3}{2}$, under the assumption that the integrator has Volterra structure, as is the case when it is the covariance function of a Gaussian process with the same regularity as the integrand; see \cite{fv2010a}. The techniques we use stem from the fractional calculus, but we have generalized the analysis to deal with general Volterra operators that include the fractional operators. In the process, we discover that in place of complementary regularity, what we need is the notion of strong-bicontinuity in the H\"{o}lder sense, which allows us to extend the class of integrands included in the result beyond those of the form \eqref{productEg}. \par

The paper is organized as follows: Section 2 give a brief introduction to the preliminary information and notation that will be used throughout the sequel. In Section 3, we summarize Volterra kernels and their relation to covariance functions, and we also introduce strongly H\"{o}lder bi-continuous functions. Finally in Section 4, we state the main theorem (Theorem \ref{nualartPropNew}) and provide its proof, as well as apply it to show convergence of the Stieltjes sums of a 2D integral without the need for complementary regularity. \par 

The bulk of the content in this paper can be found in the author's doctoral dissertation \cite{lim2016}.

\section{Preliminaries}

\subsection{2D Young-Stieltjes integration}
For a function defined on $[0,T]^{2}$, $f:[0,T]^{2}\rightarrow E$ is said to
be of finite 2D $p$-variation if 
\begin{align} \label{2DpvarBound}
\left\| f\right\|_{p-var;[0,T]^{2}}:=\sup_{\pi }\left( \sum_{i,j}\left\Vert
f \begin{pmatrix}
u_{i},u_{i+1} \\ 
v_{j},v_{j+1}
\end{pmatrix}
\right\Vert _{E}^{p}\right) ^{\frac{1}{p}}<\infty,
\end{align}
where $\pi =\left\{ \left( u_{i},v_{j}\right) \right\} $ is a partition of $[0, T]^2$, and the rectangular increment is given by 
\begin{align}  \label{rectInc}
f \begin{pmatrix}
u_{i},u_{i+1} \\ 
v_{j},v_{j+1}
\end{pmatrix}
:= f(u_{i},v_{j})+f(u_{i+1},v_{j+1})-f(u_{i},v_{j+1})-f(u_{i+1},v_{j}).
\end{align}
We will use $V^{p-var} \left(  [0, T]^{2}; E\right)$ (resp. $\mathcal{C}^{p-var} \left( [0, T]^{2}; E\right) $) to denote the set
of functions (resp. continuous functions) which satisfy \eqref{2DpvarBound}. We will use the notation 
\begin{align*}
f( \Delta_i, v) := f(u_{i+1}, v) - f(u_i, v), \\
f(u, \Delta_j) := f(u, v_{j+1}) - f(u, v_j),
\end{align*}
as well as $\norm{g}_{p-var; [s, t]}$ for the one-dimensional $p$-variation of a function $g:[0, T] \rightarrow \mathbb{R}$ over the interval $[s, t]$. \\

\begin{definition}
We say that the 2D Young-Stieltjes integral of $f$ with respect to $g$ exists if there exists a scalar $I(f, g) \in \mathbb{R}$ such that 
\begin{align}  \label{quant1}
\lim_{\left\| \pi\right\| \rightarrow 0} \left| \sum_{i, j} f\left( u_i, v_j
\right) g \begin{pmatrix}
u_i & u_{i+1} \\ 
v_j & v_{j+1}%
\end{pmatrix}
- I(f, g) \right| \rightarrow 0,
\end{align}
i.e. for each $\varepsilon > 0$, there exists a $\delta > 0$ such that for all partitions $\pi = \{ (u_i, v_j) \}$ of $[0, T]^2$ with mesh $\left\|\pi\right\| < \delta$, the quantity on the left of \eqref{quant1} is less than $\varepsilon$. In this case, we use $\int_{[0, T]^2} f \, \mathrm{d} g$ to denote $I(f, g)$, or $\int_{[s,t] \times [u, v]} f \, \mathrm{d} g$ whenever we restrict ourselves to any particular subset $[s, t] \times [u,v]$ of $[0, T]^2$.
\end{definition}

\begin{definition}
We say that $f \in V^{p-var}([s, t] \times [u,v])$ and $g \in V^{q-var}([s, t] \times [u,v])$ have complementary regularity if $\frac{1}{p} + \frac{1}{q} > 1$.
\end{definition}

The significance of this definition lies in the following theorem, which gives the existence of the Young-Stieltjes integral and Young's inequality in two dimensions; see \cite{lcl2006}, \cite{fh2014}, \cite{fv2010b} for the
one-dimensional version.

\begin{theorem} \label{2Dintegral} 
Let $f \in V^{p-var}([s, t] \times[u, v])$ and $g \in V^{q-var}([s, t] \times[u, v])$ have complementary regularity. We also assume
that $f(s, \cdot)$ and $f(\cdot, u)$ have finite p-variation, and that $f$ and $g$ have no common discontinuities. Then the 2D Young-Stieltjes integral exists and the following Young's inequality holds;
\begin{align} \label{2DYoungIneq}
\begin{split}
\left|  \int_{[s,t] \times[u,v]} f \, \mathrm{d} g\right|  \quad\leq C_{p, q}
\, {\left\vert \kern-0.25ex\left\vert \kern-0.25ex\left\vert f \right\vert
\kern-0.25ex\right\vert \kern-0.25ex\right\vert } \left\|  g\right\| _{q-var,
[s,t] \times[u,v]},
\end{split}
\end{align}
where
\begin{align*}
{\left\vert \kern-0.25ex\left\vert \kern-0.25ex\left\vert f \right\vert
\kern-0.25ex\right\vert \kern-0.25ex\right\vert } = \left|  f(s, u)\right|  +
\left\|  f(s, \cdot)\right\| _{p-var; [u, v]} + \left\|  f(\cdot, u)\right\|
_{p-var; [s, t]} + \left\|  f\right\| _{p-var, [s,t] \times[u,v]}.
\end{align*}
\end{theorem}

\begin{proof}
See \cite{towghi2002}, \cite{fv2010a}.
\end{proof}

\subsection{Volterra processes}
A Volterra kernel $K$ is a square-integrable function $K:\left[ 0,T \right]^2 \rightarrow \mathbb{R}$ such that $K(t,s)=0\;\forall s\geq t$. Associated with any Volterra kernel is a lower triangular, Hilbert-Schmidt operator $\mathbb{K}:L^{2}\left( \left[ 0,T \right] \right) \rightarrow L^{2}\left( \left[ 0,T\right] \right) $ given by 
\begin{align*}
\mathbb{K} \left( f\right) \left( \cdot \right) =\int_{0}^{T}K\left( \cdot, \, s\right)
f\left( s\right) \wrt{s} \text{ for all }f\in L^{2}\left( \left[ 0,T\right] \right).
\end{align*}
Given a standard Brownian motion $B$ and a Volterra kernel $K$, we define a Volterra process $X=\left( X_{t}\right)_{t\in \left[ 0,T\right] }$ as the It\^{o} integral 
\begin{align} \label{voltrep}
X_{t}=\int_{0}^t K(t,s)\,\mathrm{d}B_{s};  
\end{align}
this is a centered Gaussian process with covariance function \cite{udfond1999} 
\begin{align*}
R(s,t)=\int_{0}^{t\wedge s}K(t,r)K(s,r)\,\mathrm{d}r.
\end{align*}

\begin{example}
\begin{enumerate}[(i)]
\item Standard fractional Brownian motion $B^{H}$, with Hurst parameter $H\in (0,1)$, is the centered Gaussian process with covariance function 
\begin{align} \label{covfbm}
R\left( s,t\right) =\frac{1}{2}\left( s^{2H}+t^{2H}-\left\vert t-s\right\vert ^{2H}\right).  
\end{align}
It is well-known that $B^{H}$ has a Volterra representation of the form \cite{udfond1999}
\begin{align} \label{fbmKernel}
K(t, s) := \frac{1}{\Gamma (H + \frac{1}{2})} (t - s)^{H - \frac{1}{2}} F \left( H-\frac{1}{2}, \frac{1}{2} - H, H + \frac{1}{2}, 1 - \frac{t}{s} \right),
\end{align}
where $F(a,b,c,z) := \sum_{k=0}^{\infty} \frac{(a)_k (b)_k}{(c)_k k!} z^k$ is the Gauss hyper-geometric function.
\item The Riemann-Liouville process with Hurst parameter $H\in \left(0,1\right) $ is determined by the kernel $K(t,s):=C_{H}(t-s)^{H-\frac{1}{2}} \mathds{1}_{[0, t)}(s)$. Like the fractional Brownian motion, it is a self-similar process with variance $t^{2H}$; however, it does not have stationary increments.
\end{enumerate}
\end{example}

We will need the following condition on the kernel $K$.
\begin{condition} \label{amnCond}
There exists constants $C < \infty$ and $\alpha \in \left[ 0, \frac{1}{4} \right)$ such that
\begin{enumerate} [(i)]
\item $\left| K(t,s) \right| \leq C s^{-\alpha} (t - s)^{-\alpha}$ for all $0 < s < t \leq T$.
\item $\frac{\partial K(t,s)}{\partial t}$ exists for all $0 < s < t \leq T$
and satisfies $\left| \frac{\partial K(t, s)}{\partial t} \right| \leq C
\left( t-s \right)^{-(\alpha + 1)}$.
\end{enumerate}
\end{condition}
For standard fractional Brownian motion with the kernel is given by \eqref{fbmKernel}, for any $H \in (0, 1)$, we have (see Theorem 3.2 in \cite{udfond1999}) 
\begin{align} \label{kBound1}
|K (t,s)| \leq C_{1, H} \left( s^{-\left| H - \frac{1}{2} \right|} \right) (t - s)^{-\left( \frac{1}{2} - H \right)},
\end{align}
for all $0 < s < t \leq T$, and we also have 
\begin{align} \label{kBound2}
\frac{\partial K(t, s)}{\partial t} = C_{2, H} \left( \frac{t}{s} \right)^{H - \frac{1}{2}} (t - s)^{- \left(\frac{3}{2} - H \right)};
\end{align}
see \cite{ccm2003} and \cite{nualart2006}. Thus, Condition \ref{amnCond} is satisfied by the kernel as a consequence of the bounds above. 

Given a Banach space $E$ and a kernel $K$ satisfying Condition \ref{amnCond}
for some $\alpha \in \left[ 0,\frac{1}{4}\right) $, we introduce the linear
operator $\mathcal{K}^{\ast }$ (see \cite{amn2001}, \cite{dfond2005}) 
\begin{align} \label{kStarDefn}
\left( \mathcal{K}^{\ast }\phi \right) (s):=\phi \left( s\right)
K(T,s)+\int_{s}^{T}\left[ \phi \left( r\right) -\phi \left( s\right) \right]
K(\mathrm{d}r,s),
\end{align}
where the signed measure $K(\mathrm{d}r,s):=\frac{\partial K(r,s)}{\partial r}\,\mathrm{d}r$. The domain $D\left( \mathcal{K}^{\ast }\right) $ of $\mathcal{K}^{\ast }$ consists of measurable functions $\phi :\left[ 0,T\right] \rightarrow E$ for which the integral on the right-hand side exists. In particular, if $\phi $ is a $\lambda $-H\"{o}lder continuous function in the norm of $E$ for some $\lambda >\alpha $, then one can verify simply that $\phi \in D\left( \mathcal{K}^{\ast }\right)$, and $\mathcal{K}^* \phi $ is in $L^{2}([0,T];E)$. Note also that for any $a$ in $[0,T]$, $\phi \mathds{1}_{[0,a)}$ is in $D\left( \mathcal{K}^{\ast }\right) $ whenever $\phi $ is, and we have the identity 
\begin{align*}
\mathcal{K}^{\ast }\left( \phi \mathds{1}_{[0,a)}\right) (s)=\mathds{1}%
_{[0,a)}(s)\left( \phi (s)K(a,s)+\int_{s}^{a}\left[ \phi (r)-\phi (s)\right]
K(\mathrm{d}r,s)\right) .
\end{align*}

Suppose that $\phi :\left[ 0,T\right] \rightarrow \mathbb{R}$ is absolutely continuous with $\phi (0)=0$. Let $D_{0^{+}}^{1}$ be
the derivative operator $D_{0^{+}}^{1}(\phi )=\phi'$, and let $D_{T^{-}}^{1}$ denote its adjoint 
\begin{align*}
D_{T^{-}}^{1}(\phi )(t):=\phi (T)\delta _{T}(t)-\phi ^{\prime }(t),
\end{align*}
where $\delta _{T}$ is the Dirac mass at $T$. Using integration-by-parts, we have 
\begin{align*}
\left( \mathcal{K}^* \phi \right) (s)& =\phi (T)K(T,s)-\int_{s}^{T}\phi'(r)K(r,s)\,\mathrm{d}r \\
& =\left( \left( \mathbb{K}^* \circ D_{T^{-}}^{1}\right) \phi \right) (s),
\end{align*}
and note that the adjoint $\mathcal{K}$ of $\mathcal{K}^*$ is equal to $D_{0^{+}}^{1}\circ \mathbb{K}$. Thus, if $K(t,s)=\frac{1}{\Gamma (1-\alpha )}(t-s)^{-\alpha }$ as in the case of the Riemann-Liouville process, then $\mathcal{K}$ is equal to the fractional derivative $D_{0^{+}}^{\alpha }$, and $\mathcal{K}^*$ is its adjoint $D_{T^{-}}^{\alpha }$, where \eqref{kStarDefn} is written in Marchaud form; cf, \cite{skm1993}. \\

\begin{proposition}	\label{nualartProp} 
Let $\left( E,\left\Vert \cdot \right\Vert _{E}\right)$ be a Banach space and $K:\left[ 0,T\right] ^{2}\rightarrow \mathbb{R}$ be a	kernel satisfying Condition \ref{amnCond} for some $\alpha \in \left[ 0, \frac{1}{4}\right) $. Let $\phi :[0,T]\rightarrow E$ be $\lambda $-H\"{o}lder continuous, i.e. there exists $C<\infty $ such that 
\begin{align*}
\left\Vert \phi (t_{1})-\phi (t_{2})\right\Vert _{E}\leq
C\,|t_{1}-t_{2}|^{\lambda },\quad \forall t_{1},t_{2}\in \lbrack 0,T],
\end{align*}
and for any partition $\pi =\{s_{i}\}$ of $[0,T]$, let $\phi ^{\pi}:[0,T]\rightarrow E$ denote 
\begin{align*}
\phi ^{\pi }(s)=\sum_{i}\phi (s_{i})\mathds{1}_{[s_{i},s_{i+1})}(s).
\end{align*}
Then if $\lambda >\alpha $, we have 
\begin{equation*}
\lim_{\left\Vert \pi \right\Vert \rightarrow 0}\int_{0}^{T}\left\Vert \mathcal{K}^{\ast }\left( \phi ^{\pi }-\phi \right) (s) \right\Vert _{E}^{2}\, \mathrm{d}s = 0,
\end{equation*}
where $\mathcal{K}^{\ast }$ is defined as in \eqref{kStarDefn}.
\end{proposition}

\begin{proof}
The proof follows that of Proposition 8 in \cite{amn2000}. We reproduce it as the estimates derived here will be required later, and also because there is a minor error in the original proof. \par
With $\Delta_i$ denoting $\left[s_i, s_{i+1}\right)$, for all $s \in (0, T)$, we have
\begin{align*}
&\norm{ \mathcal{K}^* \left( \phi^{\pi} - \phi \right) (s) }_E \\
&\qquad = \left\| \left( \phi^{\pi}(s) - \phi(s) \right) K(T, s)  + \int_s^T \left[ \phi^{\pi}(r) - \phi(r) - (\phi^{\pi}(s) - \phi(s) ) \right] K(\mathrm{d}r, s) \right\|_E \\
&\qquad = \left\| \sum_i \mathds{1}_{\Delta_i} (s) \left( \left( \phi(s_i) - \phi(s) \right) K(T, s) + \int_s^{s_{i+1}} \left[ \phi(s_i) - \phi(r) - (\phi(s_i) - \phi(s)) \right] K(\mathrm{d}r, s)  \right. \right. \\
&\qquad \qquad \qquad \qquad + \left. \left. \sum_{k \geq i+1} \int_{s_k}^{s_{k+1}} \left[ \phi(s_k) - \phi(r) - (\phi(s_i) - \phi(s)) \right] K(\mathrm{d}r, s) \right) \right\|_E \\
&\qquad \leq I_1(s) + I_2(s) + I_3(s) + I_4 (s),
\end{align*}
where
\begin{align} \label{IIntegrals}
\begin{split}
I_1 (s) &:= \sum_i \mathds{1}_{\Delta_i} (s) \left\| \phi(s_i) - \phi(s) \right\|_E \abs{K(T, s)}
\leq C \sum_i \mathds{1}_{\Delta_i} (s) \underset{\displaystyle := {\xi_{i, 1} (s)}}{\underbrace{(s - s_i)^{\lambda} s^{-\alpha} (T-s)^{-\alpha}}}, \\
I_2 (s) &:= \sum_i \mathds{1}_{\Delta_i} (s) \int_s^{s_{i+1}} \left\|\phi(s) - \phi(r) \right\|_E \abs{\pd{K(r, s)}{r}} \wrt{r}
\leq C \sum_i \mathds{1}_{\Delta_i} (s) \underset{\displaystyle := \xi_{i, 2}(s)}{\underbrace{(s_{i+1} - s)^{\lambda - \alpha}}}, \\
I_3(s) &:= \sum_i \mathds{1}_{\Delta_i} (s) \sum_{k \geq i+1} \int_{s_k}^{s_{k+1}} \left\|\phi(s_k) - \phi(r) \right\|_E \abs{\pd{K(r, s)}{r}} \wrt{r} \\
&\leq C \sum_i \mathds{1}_{\Delta_i} (s) \underset{\displaystyle := \xi_{i, 3}(s)}{\underbrace{\sum_{k \geq i+1} \int_{s_k}^{s_{k+1}} (r-s_k)^{\lambda} (r - s)^{-\alpha - 1} \wrt{r}}}, \\
I_4 (s) &:= \sum_i \mathds{1}_{\Delta_i} (s) \norm{\phi(s_i) - \phi(s)}_E \sum_{k \geq i+1} \int_{s_k}^{s_{k+1}} \abs{\pd{K(r, s)}{r}} \wrt{r}, \\
&\leq C \sum_i \mathds{1}_{\Delta_i} (s) \underset{\displaystyle := \xi_{i,4}(s)}{\underbrace{(s - s_i)^{\lambda} \left((s_{i+1} - s)^{-\alpha} - (T-s)^{-\alpha} \right)}}.
\end{split}
\end{align}
For the second term, we have
\begin{align*}
\int_0^T \left[ \sum_i \mathds{1}_{\Delta_i} (s) \xi_{i, 2} (s) \right]^2 \wrt{s} 
&\leq C \sum_i \int_{s_i}^{s_{i+1}} (s_{i+1} - s)^{2(\lambda - \alpha)} \wrt{s} \\
&\leq C \sum_i \abs{s_{i+1} - s_i}^{2(\lambda - \alpha) + 1} \\
&\leq C \norm{\pi}^{2(\lambda - \alpha)} \rightarrow 0.
\end{align*}
For the third term, since $r - s_k \leq r - s$ if $s \in [s_i, s_{i+1})$ and $s_k \geq s_{i+1}$, we have
\begin{align} \label{I3estimate}
\begin{split}
\xi_{i,3} (s)
&= \sum_{k \geq i+1} \int_{s_k}^{s_{k+1}} (r - s_k)^{\lambda - \alpha - \varepsilon} (r - s_k)^{\alpha + \varepsilon} (r - s)^{- \alpha - 1} \wrt{r} \\
&\leq \norm{\pi}^{\lambda - \alpha - \varepsilon} \sum_{k \geq i+1} \int_{s_k}^{s_{k+1}} (r - s)^{\varepsilon - 1} \wrt{r} \\
&\leq \frac{(T - s)^{\varepsilon}}{\varepsilon} \norm{\pi}^{\lambda - \alpha - \varepsilon},
\end{split}
\end{align}
where $\varepsilon$ is chosen such that $\lambda - \alpha - \varepsilon > 0$. Hence
\begin{align*}
\int_0^T \left[ \sum_i \mathds{1}_{\Delta_i} (s) \xi_{i, 3} (s) \right]^2 \wrt{s} 
&\leq C \norm{\pi}^{2(\lambda - \alpha - \varepsilon)} \int_0^T (T - s)^{2\varepsilon} \wrt{s} \\
&\leq C \, T^{1 + 2\varepsilon} \norm{\pi}^{2(\lambda - \alpha - \varepsilon)} \rightarrow 0.
\end{align*}
Finally for the first and fourth terms, we have
\begin{align} \label{I4estimate}
\begin{split}
\xi_{i, 1} (s) 
&\leq C \, (s - s_i)^{\lambda} (s_{i+1} - s)^{-\alpha} s^{-\alpha}, \quad \mathrm{and} \\
\xi_{i, 4} (s)
&\leq C \, (s - s_i)^{\lambda} (s_{i+1} - s)^{-\alpha}.
\end{split}
\end{align}
Thus for $k = 1, 4$, we obtain
\begin{align} \label{xiLimit}
\begin{split}
\int_0^T \left[ \sum_i \mathds{1}_{\Delta_i} (s) \xi_{i, k} (s) \right]^2 \wrt{s}
&\leq C \left( \sum_i \int_{s_i}^{s_{i+1}} (s - s_i)^{4\lambda} (s_{i+1} - s)^{-4\alpha} \wrt{s} \right)^{\frac{1}{2}} \left( \int_0^T s^{-4\alpha} \wrt{s} \right)^{\frac{1}{2}} \\
&\leq C \sqrt{B(4\lambda + 1, 1 - 4\alpha)} \sqrt{ \sum_i (s_{i+1} - s_i)^{4(\lambda - \alpha) + 1} }\\
&\leq C \left\| \pi \right\|^{2(\lambda - \alpha)} \rightarrow 0,
\end{split}
\end{align}
where $B(\cdot, \cdot)$ in the second line denotes the beta function and is well-defined since $4\lambda + 1 > 0$ and $1 - 4\alpha > 0$.
\end{proof}

\section{The operator $\mathcal{K}^* \otimes \mathcal{K}^*$ and H\"{o}lder bi-continuity}
In this section, we will continue to use $E$ to denote a general Banach space with norm $\left\Vert \cdot \right\Vert _{E}$.

\begin{definition} \label{kStarTensorOp}
Let $\mathcal{K}^* \otimes \mathcal{K}^*$ denote the following operator, 
\begin{align*}
(\mathcal{K}^* \otimes \mathcal{K}^{\ast })\psi (u,v)
&\vcentcolon=\psi (u,v)K(T,v)K(T,u)+K(T,v)A^{K}\big(\psi (\cdot ,v)\big)(u) \\
&\qquad+ K(T,u)A^{K} \big( \psi (u,\cdot )\big)(v) + B^{K}(\psi )(u,v),
\end{align*}
where 
\begin{align*}
& A^{K}(\phi )(s):=\int_{s}^{T}\left[ \phi (r)-\phi (s)\right] K(\mathrm{d} r,s) \\
& B^{K}(\psi )(u,v):=\int_{v}^{T}\int_{u}^{T}\psi 
\begin{pmatrix}
u & r_{1} \\ 
v & r_{2}
\end{pmatrix}
K(\mathrm{d}r_{1},u)K(\mathrm{d}r_{2},v),
\end{align*}
which is defined for any measurable function $\psi :[0,T]^{2}\rightarrow E$ for which the integrals on the right side exist.
\end{definition}
This definition is motivated by the following observation: if $\psi $ has the product form $\psi \left( s,t\right) =\psi _{1}\left( s\right) \otimes	\psi _{2}\left( t\right) $ for some $\psi _{1},\psi _{2}:\left[ 0,T\right] \rightarrow F$, then by applying the previous definition with $E:=F\otimes F$	(completed with respect to any cross-norm) we have 
\begin{align}
\left( \mathcal{K}^{\ast }\otimes \mathcal{K}^{\ast }\right) \psi (s,t)
= \left( \mathcal{K}^{\ast }\psi _{1}\right) (s)\otimes \left( \mathcal{K}^* \psi _{2}\right) (t),  \label{kStarProduct}
\end{align}
whenever the terms on the right side exist.

We also have the following technical lemma, which shows that the action of $\mathcal{K}^{\ast }\otimes \mathcal{K}^*$ on a function $\phi$ is the same as the iterated application of $\mathcal{K}^{\ast }$ to the first, then second variable (or vice-versa) of $\phi $. 

\begin{lemma} \label{iteratedKStar}
Let $\phi:[0, T]^2 \rightarrow E$, and let $g_s(t)$ denote $\mathcal{K}^* \left( \phi( \cdot, t) \right) (s)$. If $g_s \in D \left( \mathcal{K}^* \right)$, then
\begin{align*}
\mathcal{K}^* \left( g_s \right) (t) = \mathcal{K}^* \otimes \mathcal{K}^* \phi (s,t)
\end{align*} 
for all $(s, t) \in [0, T]^2$ for which both sides of the equation are well-defined. Similarly, if $h_t(s) := \mathcal{K}^* \left( \phi( s, \cdot) \right) (t) \in D \left( \mathcal{K}^* \right)$, then
\begin{align*}
\mathcal{K}^* \left( h_t \right) (s) = \mathcal{K}^* \otimes \mathcal{K}^* \phi (s,t).
\end{align*} 
\end{lemma}

\begin{proof}
We have
\begin{align*} 
\mathcal{K}^* (g_s) (t) 
&= g_s(t) K(T, t) + \int_t^T \left[ g_s(r_2) - g_s(t) \right] K(\mathrm{d}r_2, t) \\
&= \left[ \phi(s, t) K(T, s) + A^K\big( \phi(\cdot, t) \big) (s) \right] K(T, t) \\
&\qquad + \int_t^T \left[ \phi (s, r_2) K(T, s) + A^K\big( \phi (\cdot, r_2)\big) (s) \right] K(\mathrm{d}r_2, t) \\
&\qquad - \int_t^T \left[ \phi (s, t) K(T, s) + A^K\big( \phi(\cdot, t) \big) (s) \right] K(\mathrm{d}r_2, t) \\
&= \mathcal{K}^* \otimes \mathcal{K}^* \phi (s, t),
\end{align*}
and the second statement can be proved by a similar computation.
\end{proof}

As with the operator $\mathcal{K}^{\ast }$, when $K(t,s)=\frac{1}{\Gamma (1-\alpha )}(t-s)^{-\alpha }$, $\mathcal{K}^{\ast }\otimes \mathcal{K}^*$ is precisely the mixed fractional derivative $D_{(T,T)^{-}}^{(\alpha,\alpha )}$ written in Marchaud form; see Chapter 24 in \cite{skm1993}. \par

Moving beyond product functions in the domain of $\mathcal{K}^{\ast }\otimes \mathcal{K}^{\ast }$, we extend the discussion to H\"{o}lder bi-continuous functions.

\begin{definition}
Let $0<\lambda \leq 1.$ We say that a function $\phi :[0,T]^{2}\rightarrow E$ is $\lambda $-H\"{o}lder bi-continuous in the norm of $E$ (or simply $\lambda $-H\"{o}lder bi-continuous in the case where $E$ is	finite-dimensional), if for all $u_{1},u_{2},v_{1},v_{2}\in [0,T]$ we have 
\begin{align*}
&\sup_{v\in \left[ 0,T\right] }\left\Vert \phi (u_{2},v)-\phi (u_{2},v)\right\Vert _{E}\leq C\,\left\vert u_{2}-u_{1}\right\vert ^{\lambda}, \\
&\sup_{u\in \left[ 0,T\right] }\left\Vert \phi (u,v_{2}) - \phi(u,v_{1})\right\Vert _{E}\leq C\,\left\vert v_{2}-v_{1}\right\vert^{\lambda}.
\end{align*}	
\end{definition}

\begin{definition}
Let $0<\lambda \leq 1.$ We say that a function $\phi :[0,T]^{2}\rightarrow E$ is strongly $\lambda $-H\"{o}lder bi-continuous in the norm of $E$, if for all $u_{1},u_{2},v_{1},v_{2}\in [0,T]$ we have 
\begin{align*}
\sup_{v\in \left[ 0,T\right] }\left\Vert \phi (u_{2},v)-\phi (u_{2},v)\right\Vert _{E}\leq C\,\left\vert u_{2}-u_{1}\right\vert ^{\lambda},
\quad \sup_{u\in \left[ 0,T\right] }\left\Vert \phi (u,v_{2}) - \phi(u,v_{1})\right\Vert _{E}\leq C\,\left\vert v_{2}-v_{1}\right\vert^{\lambda},
\end{align*}
and 
\begin{align*}
\left\Vert \phi 
\begin{pmatrix}
u_{1} & u_{2} \\ 
v_{1} & v_{2}
\end{pmatrix}
\right\Vert _{E}\leq C\,\left\vert u_{2}-u_{1}\right\vert ^{\lambda}\left\vert v_{2}-v_{1}\right\vert ^{\lambda}.
\end{align*}
\end{definition}

We have the following simple lemmas.

\begin{lemma} \label{strongToWeak}
If a function $\phi$ is $\lambda$-H\"{o}lder bi-continuous, then it is strongly $\frac{\lambda}{2}$-H\"{o}lder bi-continuous.
\end{lemma}
\begin{proof}
For any $u_1, u_2, v_1, v_2 \in [0, T]$, we have
\begin{align} \label{phiEqA}
\left\Vert \phi 
\begin{pmatrix}
u_{1} & u_{2} \\ 
v_{1} & v_{2}
\end{pmatrix}
\right\Vert _{E}\leq C\, \norm{u_2 -u_1}^{\lambda}
\end{align}
and
\begin{align} \label{phiEqB}
\left\Vert \phi 
\begin{pmatrix}
u_{1} & u_{2} \\ 
v_{1} & v_{2}
\end{pmatrix}
\right\Vert _{E}\leq C\, \norm{v_2- v_1}^{\lambda}.
\end{align}
Now taking square roots on both sides of \eqref{phiEqA} and \eqref{phiEqB} and multiplying them together gives the proof.
\end{proof}

\begin{lemma} \label{productHolder} 
Let $\psi :[0,T]\rightarrow E$ and $\phi:[0,T]\rightarrow \mathcal{L}(E;E)$, where $\mathcal{L}(E;E)$ denotes the space of bounded linear operators from $E$ to $E$ equipped with the operator norm. Assume that $\psi $ and $\phi $ are $\lambda $-H\"{o}lder continuous in the norm of $E$ and $\mathcal{L}(E;E)$ respectively, and let $\phi \circ	\psi :[0,T]^{2}\rightarrow E$ be the function given by 
\begin{align*}
\left[ \phi \circ \psi \right] \left( s,t\right) :=\phi \left( t\right) 
\left[ \psi \left( s\right) \right] \text{, \ for }\left( s,t\right) \in [0,T]^{2}.
\end{align*}
Then $\phi \circ \psi $ is strongly $\lambda $-H\"{o}lder bi-continuous in the norm of $E$.
\end{lemma}

\begin{proof}
Let $\norm{\phi}_{\infty} := \sup_{t \in [0, T]} \norm{\phi(t)}_{\mathcal{L}(E;E)}$ and $\norm{\psi}_{\infty} := \sup_{t \in [0, T]} \norm{\psi(t)}_E$ denote the supremum norms of $\phi$ and $\psi$ respectively. \par
We have
\begin{align*}
\norm{ \phi (\psi) (u, v_2) - \phi (\psi) (u, v_1) }_E
&= \norm{ \phi(u) \left( \psi(v_2) - \psi(v_1) \right)}_E \\
&\leq \norm{\phi(u)}_{\mathcal{L}(E;E)} \norm{\psi(v_2) - \psi(v_1)}_E \\
&\leq C \norm{\phi}_{\infty} \abs{v_2 - v_1}^{\lambda},
\end{align*}
\begin{align*}
\norm{ \phi (\psi) (u_2, v) - \phi (\psi) (u_1, v) }_E
&= \norm{ \left( \phi(u_2) -\phi(u_1) \right) \left( \psi(v) \right)}_E \\
&\leq \norm{\phi(u_2) - \phi(u_1)}_{\mathcal{L}(E;E)} \norm{\psi(v)}_E \\
&\leq C \norm{\psi}_{\infty} \abs{u_2 - u_1}^{\lambda},
\end{align*}
and
\begin{align*}
\norm{ \phi (\psi) \begin{pmatrix}
	u_1 & u_2 \\
	v_1 & v_2
	\end{pmatrix} }_E
&= \norm{ \phi(u_1) \left( \psi(v_1) \right) + \phi(u_2) \left( \psi(v_2) \right) - \phi(u_1) \left( \psi(v_2) \right) - \phi(u_2) \left(\psi(v_1) \right) }_E \\
&\leq \norm{\left( \phi(u_2) - \phi(u_1) \right)}_{\mathcal{L}(E;E)} \norm{ \left(\psi(v_2) - \psi(v_1) \right) }_E \\
&\leq C \abs{u_2 - u_1}^{\lambda} \abs{v_2 - v_1}^{\lambda}.
\end{align*}
\end{proof}

\begin{lemma}
Let $\mathcal{K}^{\ast }\otimes \mathcal{K}^{\ast }$ denote the operator in	Definition \ref{kStarTensorOp}, and assume the kernel satisfies Condition \ref{amnCond} for some $\alpha \in \lbrack 0,\frac{1}{4})$. If $\psi:[0,T]^{2}\rightarrow E$ is strongly $\lambda $-H\"{o}lder bi-continuous in the norm of $E$ and $\lambda >\alpha $, then 
\begin{align*}
\int_{[0,T]^{2}}\left\Vert \left( \mathcal{K}^{\ast }\otimes \mathcal{K}^* \right) (\psi )(u,v)\right\Vert _{E}^{4}\,\mathrm{d}u\,\mathrm{d}v < \infty,
\end{align*}
and 
\begin{align*}
\int_{0}^{T}\left\Vert \left( \mathcal{K}^{\ast }\otimes \mathcal{K}^* \right) (\psi )(r,r)\right\Vert_{E}^{2}, \mathrm{d}r < \infty.
\end{align*}
\end{lemma}

\begin{proof}
It follows that for $\lambda > \alpha$, because of Condition \ref{amnCond} and strong H\"{o}lder bi-continuity, we have 
\begin{align*}
\begin{split}
\norm{\psi(u, v) K(T, u) K(T, v)}_E &\leq \frac{C}{v^{\alpha} (T - v)^{\alpha} u^{\alpha} (T - u)^{\alpha}}, \\
\norm{K(T, v) A^K (\psi(\cdot, v) ) (u)}_E &\leq C \frac{(T - u)^{\lambda - \alpha}}{v^{\alpha} (T - v)^{\alpha}}, \\
\norm{K(T, u) A^K (\psi(u, \cdot) ) (v)}_E &\leq C \frac{(T - v)^{\lambda - \alpha}}{u^{\alpha} (T - u)^{\alpha}}
\end{split}
\end{align*}
and
\begin{align*}
\norm{B^K (\psi) (u, v)}_E &\leq C (T- u)^{\lambda - \alpha} (T - v)^{\lambda - \alpha}
\end{align*}
for all $(u, v) \in (0, T)^2$. The proof is complete when we use the fact that $\alpha < \frac{1}{4}$.
\end{proof}

\section{Main result}
We now arrive at the main result of the paper, which demonstrates that complementary regularity between the integrand and integrator is a sufficient but not necessary condition for the multi-dimensional Young-Stieltjes integral to exist. The following theorem is key to showing the main result, but due to its length, we will defer its proof until the end of the section.
\begin{theorem} \label{nualartPropNew} 
Given a Banach space $E$, let $\psi :[0,T]^{2}\rightarrow E$ be a function which is strongly $\lambda $-H\"{o}lder bi-continuous in the norm of $E$. For any partition $\{(u_{i},v_{j})\}$ of $[0,T]^{2}$, let $\psi ^{\pi}:[0,T]^{2}\rightarrow E$ denote 
\begin{align*}
\psi ^{\pi }(u,v):=\sum_{i,j}\psi (u_{i},v_{j})\mathds{1}_{[u_{i},u_{i+1})}(u)\mathds{1}_{[v_{j},v_{j+1})}(v).
\end{align*}
In addition, let $\mathcal{K}^{\ast }\otimes \mathcal{K}^{\ast }$ denote the operator in Definition \ref{kStarTensorOp}, where the Volterra kernel $K$ satisfies Condition \ref{amnCond} for some $\alpha \in \left[ 0,\frac{1}{4} \right)$. Then if $\lambda >\alpha $, we have 
\begin{align} \label{result1}
\lim_{\left\Vert \pi \right\Vert \rightarrow 0}\int_{[0,T]^{2}}\left\Vert \left( \mathcal{K}^{\ast }\otimes \mathcal{K}^{\ast }\left( \psi ^{\pi}-\psi \right) \right) (u,v)\right\Vert_{E}^{2}\,\mathrm{d}u\,\mathrm{d}v = 0,
\end{align}
and 
\begin{align} \label{result2}
\lim_{\left\Vert \pi \right\Vert \rightarrow 0}\int_{0}^{T}\left\Vert \left( \mathcal{K}^{\ast }\otimes \mathcal{K}^{\ast }\left( \psi ^{\pi } - \psi \right) \right) (r,r)\right\Vert _{E}\,\mathrm{d}r=0.
\end{align}
\end{theorem}
The following corollary is an immediate consequence of the preceding theorem and Lemma \ref{strongToWeak}.
\begin{corollary} \label{npNewCoro} 
Let $\psi :[0,T]^{2}\rightarrow E$ be $\lambda $-H\"{o}lder bi-continuous in the norm of $E$. In addition, let $\mathcal{K}^* \otimes \mathcal{K}^*$ denote the operator in Definition \ref{kStarTensorOp}, where the Volterra kernel $K$ satisfies Condition \ref{amnCond} for some $\alpha \in \left[ 0,\frac{1}{4} \right)$. Then both \eqref{result1} and \eqref{result2} in Theorem \ref{nualartPropNew} remain true if $\lambda > 2\alpha$.
\end{corollary}

We now prove the following proposition.
\begin{proposition}	\label{kStarRProp} 
Let $R$ be the covariance function of a Volterra process with kernel satisfying Condition \ref{amnCond} for some $\alpha \in [0,\frac{1}{4})$. If $\phi :[0,T]^{2}\rightarrow \mathbb{R}$ is a strongly $\lambda $-H\"{o}lder bi-continuous function with $\lambda >\alpha$, or a $\lambda $-H\"{o}lder bi-continuous function with $\lambda > 2\alpha$, then the Young integral $\int_{[0,T]^{2}}\phi (s,t)\,\mathrm{d}R(s,t)$ exists, and we have 
\begin{align*}
\int_{[0,T]^{2}}\phi (u,v)\,\mathrm{d}R(u,v)=\int_{0}^{T}\mathcal{K}^{\ast }\otimes \mathcal{K}^* \phi (r,r)\,\mathrm{d}r.
\end{align*}
\end{proposition}

\begin{proof}
Fixing an arbitrary partition of $[0, T]^2$ as $\pi = \left\{ \left( u_i, v_j \right) \right\}$, we denote $\phi^{\pi}$ as
\begin{align*}
\phi^{\pi} (u, v) := \sum_{i, j} \phi(u_i, v_j) \mathds{1}_{[u_i, u_{i+1})} (u) \mathds{1}_{[v_j, v_{j+1})} (v).
\end{align*}
We have
\begin{align} \label{kFirst}
\begin{split}
\int_{[0, T]^2} \phi^{\pi} (u, v) \wrt{R(u, v)} 
&=\sum_{i, j} \phi(u_i, v_j) R \begin{pmatrix}
u_i & u_{i+1} \\
v_j & v_{j+1}
\end{pmatrix} \\
&= \sum_{i, j} \phi(u_i, v_j) \int_0^T K(\Delta_i, r) K(\Delta_j, r) \wrt{r} \\
&= \int_0^T \sum_{i,j} \mathcal{K}^* \left( \phi(u_i, v_j) \mathds{1}_{[u_i, u_{i+1})} \right) (r) \, \mathcal{K}^* \left( \mathds{1}_{[v_j, v_{j+1})} \right) (r) \wrt{r} \\
&= \int_0^T \mathcal{K}^* \otimes \mathcal{K}^* \phi^{\pi} (r, r) \wrt{r}.
\end{split}
\end{align}
By Theorem \ref{nualartPropNew}, this converges to $\int_0^T \mathcal{K}^* \otimes \mathcal{K}^* \phi (r, r) \wrt{r}$ as the mesh of the partition goes to zero.
\end{proof}

As a particular application of this result, we can consider the case where $\phi $ is a product of two single-variable $\lambda $-H\"{o}lder continuous functions $\phi _{1}$ and $\phi _{2}$. We then have
\begin{align*}
\int_{[0,T]^{2}}\phi (s,t)\,\mathrm{d}R(s,t)& =\int_{[0,T]^{2}}\phi_{1}(s)\phi _{2}(t)\,\mathrm{d}R(s,t) \\
& =\int_{0}^{T} \mathcal{K}^* \otimes \mathcal{K}^* (\phi_{1}\,\phi_{2}) (r,r) \,\mathrm{d}r \\
& =\int_{0}^{T}\mathcal{K}^* \phi_{1}(r) \mathcal{K}^* \phi_{2}(r)\,\mathrm{d}r.
\end{align*}

Another application of Theorem \ref{nualartPropNew} is in proving the existence of iterated 2D Young integrals.

\begin{proposition}	\label{kStarRProp2} 
Let $R$ be the covariance function of a Volterra process with kernel satisfying Condition \ref{amnCond} for some $\alpha \in \left[ 0,\frac{1}{4} \right)$. If $\psi _{1}, \psi _{2}:[0,T]^{2}\rightarrow \mathbb{R}$ are both strongly $\lambda $-H\"{o}lder bi-continuous functions with $\lambda >\alpha$, or both $\lambda $-H\"{o}lder bi-continuous functions with $\lambda > 2\alpha$, then the Young integral 
\begin{align*}
\int_{[0,T]^{2}} \int_{[0,T]^{2}}\psi _{1}(q,s)\,\psi _{2}(r,t) \mathrm{d}R(q,r)\,\mathrm{d}R(s,t)
\end{align*}
exists, and is equal to 
\begin{align} \label{kkExp}
\int_{[0,T]^{2}} \mathcal{K}^* \otimes \mathcal{K}^* \, \psi_{1}(r_{1},r_{2}) \, \mathcal{K}^* \otimes \mathcal{K}^* \, \psi_{2}(r_{1},r_{2}) \, \mathrm{d}r_{1} \,\mathrm{d}r_{2}.
\end{align}
\end{proposition}

\begin{proof}
Denoting $\psi_1^{\pi_1}$ and $\psi_2^{\pi_2}$ as
\begin{align*}
&\psi_1^{\pi_1} (q, s) := \sum_{i, j} \psi_1(q_i, s_j) \mathds{1}_{[q_i, q_{i+1})} (q) \mathds{1}_{[s_j, s_{j+1})} (s), \\
&\psi_2^{\pi_2} (r, t) := \sum_{k, l} \psi_2(r_k, t_l) \mathds{1}_{[r_k, r_{k+1})} (r) \mathds{1}_{[t_l, t_{l+1})} (t),
\end{align*}
where $\pi_1 \times \pi_2 = \left\{ \left( q_i, s_j \right) \times \left(r_k, t_l \right) \right\}$ is an arbitrary partition of $[0, T]^2 \times [0, T]^2$, we have
\begin{align} \label{kSecond}
\begin{split}
&\int_{[0, T]^2} \int_{[0, T]^2} \psi_1^{\pi_1} (q, s) \, \psi_2^{\pi_2} (r, t) \wrt{R(q, r)} \wrt{R(s, t)} \\
&\qquad \qquad= \sum_{i, j ,k, l} \psi_1 (q_i, s_j) \psi_2 (r_k, t_l) R \begin{pmatrix}
q_i & q_{i+1} \\
r_k & r_{k+1}
\end{pmatrix} R \begin{pmatrix}
s_j & s_{j+1} \\
t_l & t_{l+1}
\end{pmatrix} \\
&\qquad \qquad= \sum_{i, j, k, l} \psi_1 (q_i, s_j) \psi_2 (r_k, t_l) \int_0^T K(\Delta_i, r_1) K(\Delta_k, r_1) \wrt{r_1} \int_0^T K(\Delta_j, r_2) K(\Delta_l, r_2) \wrt{r_2} \\
&\qquad \qquad= \int_0^T \int_0^T \sum_{i, j} \psi_1(q_i, s_j) K(\Delta_i, r_1) K(\Delta_j, r_2) \sum_{k, l} \psi_2 (r_k, t_l) K(\Delta_k, r_1) K(\Delta_l, r_2) \wrt{r_1} \wrt{r_2} \\
&\qquad \qquad= \int_{[0, T]^2} \mathcal{K}^* \otimes \mathcal{K}^* \, \psi_1^{\pi_1} (r_1, r_2) \, \mathcal{K}^* \otimes \mathcal{K}^* \, \psi_2^{\pi_2} (r_1, r_2) \wrt{r_1} \wrt{r_2}.
\end{split}
\end{align}
Using Theorem \ref{nualartPropNew}, and the continuity of the inner product \\
$\left\langle f, g \right\rangle := \int_{[0, T]^2} f(r_1, r_2) g(r_1, r_2) \wrt{r_1} \wrt{r_2}$, we see that the above expression converges to \eqref{kkExp}.
\end{proof} 

\begin{proof} [\bf Proof of Theorem \ref{nualartPropNew}]
(i) From the definition of $\mathcal{K}^* \otimes \mathcal{K}^*$, for all $(u, v) \in (0, T)^2$, we have
\begin{align} \label{estimateC}
\begin{split}
\mathcal{K}^* \otimes \mathcal{K}^* \left( \psi^{\pi} - \psi \right) (u, v)
&= \left( \psi^{\pi }(u, v) - \psi(u, v) \right) K(T, v) K(T, u) + K(T, v) A^K \big( (\psi^{\pi} - \psi) (\cdot, v) \big) (u) \\
&\qquad + K(T, u) A^K \big( (\psi^{\pi} - \psi) (u, \cdot) \big) (v) + B^K (\psi^{\pi} - \psi) (u, v).
\end{split}
\end{align}
We will examine each of the terms on the right side of the above expression. To bound the terms, we will use the functions $\xi_{i, k}$, $k = 1, 2, 3, 4$, in Proposition \ref{nualartProp}, which are defined for any partition $\pi = \{r_i\}$ of $[0, T]$. In particular, recall that for $k = 1, 2, 3, 4$,
\begin{align} \label{xiConv}
\int_0^T \left[ \sum_i \mathds{1}_{\Delta_i} (r) \xi_{i, k} (r) \right]^2 \wrt{r} \xrightarrow{\norm{\pi} \rightarrow 0} 0, 
\end{align}
and here we use $\Delta_i$ to denote $\left[ r_i, r_{i+1} \right)$. \par
For the first term on the right of \eqref{estimateC}, we have
\begin{align} \label{A1estimate}
\begin{split}
&\norm{ \left( \psi^{\pi}(u, v) - \psi(u, v) \right) K(T, v) K(T, u) }_E \\
&\qquad \qquad \qquad \leq \norm{ K(T, u) K(T, v) \sum_{i, j} \mathds{1}_{\Delta_i} (u) \mathds{1}_{\Delta_j} (v) \left[ \psi(u_i, v_j) - \psi(u, v) \right] }_E \\
&\qquad \qquad \qquad \leq \norm{ K(T, u) K(T, v) \sum_{i, j} \mathds{1}_{\Delta_i} (u) \mathds{1}_{\Delta_j} (v) \left[ \psi(u_i, v_j) - \psi(u, v_j) + \psi(u, v_j) - \psi(u, v) \right] }_E \\
&\qquad \qquad \qquad \leq C\left( \frac{1}{v^{\alpha} (T-v)^{\alpha}} \sum_i \mathds{1}_{\Delta_i} (u) \xi_{i, 1} (u) + \frac{1}{u^{\alpha} (T-u)^{\alpha}} \sum_j \mathds{1}_{\Delta_j} (v) \xi_{j, 1}(v) \right),
\end{split}
\end{align}
and thus
\begin{align*}
\int_{[0, T]^2} \norm{ \left( \psi^{\pi}(u, v) - \psi(u, v) \right) K(T, v) K(T, u) }_E^2 \wrt{u} \wrt{v} \rightarrow 0
\end{align*}
For the second term, we have
\begin{align} \label{A2estimate}
\norm{ K(T, v) A^K\big((\psi^{\pi} - \psi) (\cdot, v)\big) (u) }_E 
\leq \frac{C}{v^{\alpha} (T-v)^{\alpha}} \sum_i \mathds{1}_{\Delta_i} (u) \left[ \xi_{i, 2}(u) + \xi_{i, 3}(u) + \xi_{i, 4}(u) \right],
\end{align}
and this yields
\begin{align*}
\int_{[0, T]^2} \norm{ K(T, v) A^K\big((\psi^{\pi} - \psi) (\cdot, v)\big) (u) }_E^2 \wrt{u} \wrt{v} \rightarrow 0
\end{align*}
as the mesh of the partition vanishes. The third term is handled similarly by interchanging $u$ and $v$. \par
	
The last term $B^K(\psi^{\pi} - \psi) (u, v) $ can be written as
\begin{align} \label{estimateB}
\begin{split}
&\sum_{i, j} \mathds{1}_{\Delta_i}(u) \mathds{1}_{\Delta_j}(v)
\left( \int_u^{u_{i+1}} \int_v^T f(r_1, r_2) \wrt{r_2}\wrt{r_1} 
+ \int_v^{v_{j+1}} \int_{u_{i+1}}^T f(r_1, r_2) \wrt{r_1} \wrt{r_2} \right. \\
&\fqquad \left. + \sum_{\substack{k \geq i+1 \\ l \geq j+1}} \int_{u_k}^{u_{k+1}} \int_{v_l}^{v_{l+1}} f(r_1, r_2) \wrt{r_2}\wrt{r_1} \right),
\end{split}
\end{align}
where $f$ denotes (suppressing the dependence on $u, v$ in the notation)
\begin{align*}
f(r_1, r_2)
:= \left( \psi^{\pi} \begin{pmatrix}
u & r_1 \\
v & r_2
\end{pmatrix} - \psi \begin{pmatrix}
u & r_1 \\
v & r_2
\end{pmatrix} \right) \pd{K(r_2, v)}{r_2} \pd{K(r_1, u)}{r_1}.
\end{align*}
For the next few estimates, we will fix $i$ and $j$ and consider each summand in \eqref{estimateB} term by term. Note that the presence of $\mathds{1}_{\Delta_i} (u) \mathds{1}_{\Delta_j} (v)$ means that we can take $u$ and $v$ to lie in $[u_i, u_{i+1})$ and $[v_j, v_{j+1})$ respectively.

For the first term, since $\psi^{\pi} \begin{pmatrix}
u & r_1 \\
v & r_2
\end{pmatrix} = \psi^{\pi} \begin{pmatrix}
u_i & u_i \\
v & r_2
\end{pmatrix} = 0$ on $[u, u_{i+1}) \times [v, T)$, we have
\begin{align} \label{newEstimate1}
\begin{split}
&\norm{ \mathds{1}_{\Delta_i} (u) \mathds{1}_{\Delta_j} (v) \int_u^{u_{i+1}} \int_v^T f(r_1, r_2) \wrt{r_2}\wrt{r_1} }_E \\
&\fqquad= \norm{ \mathds{1}_{\Delta_i} (u) \mathds{1}_{\Delta_j} (v) \int_u^{u_{i+1}} \left( \int_v^T - \psi \begin{pmatrix}
u & r_1 \\
v & r_2
\end{pmatrix} K(\mathrm{d}r_2, v) \right) K(\mathrm{d}r_1, u)}_E \\
&\fqquad\leq C \, \left| u_{i+1} - u \right|^{\lambda - \alpha} \left| T - v \right|^{\lambda - \alpha},
\end{split}
\end{align}
where we use Condition \ref{amnCond} and the strong H\"{o}lder bi-continuity of $\psi$ to obtain the estimate. \par
Similarly for the second term in \eqref{estimateB}, we have
\begin{align} \label{newEstimate2}
\begin{split}
&\norm{ \mathds{1}_{\Delta_i} (u) \mathds{1}_{\Delta_j} (v) \int_v^{v_{j+1}} \int_{u_{i+1}}^T f(r_1, r_2) \wrt{r_1} \wrt{r_2} }_E \\
&\fqquad= \norm{ \mathds{1}_{\Delta_i} (u) \mathds{1}_{\Delta_j} (v) \int_v^{v_{j+1}} \left( \int_{u_{i+1}}^T - \psi \begin{pmatrix}
u & r_1 \\
v & r_2
\end{pmatrix} K(\mathrm{d}r_2, v) \right) K(\mathrm{d}r_1, u)}_E \\
&\fqquad\leq C \, \left| T - u \right|^{\lambda - \alpha} \left| v_{j+1} - v \right|^{\lambda - \alpha},
\end{split}
\end{align}
For the last term in \eqref{estimateB}, note that when $(r_1, r_2) \in [u_k, u_{k+1}) \times [v_l, v_{l+1})$, we have
\begin{align*}
&\mathds{1}_{\Delta_i} (u) \mathds{1}_{\Delta_j} (v) \int_{u_k}^{u_{k+1}} \int_{v_l}^{v_{l+1}} f(r_1, r_2) \wrt{r_2} \wrt{r_1} \\
&\qquad \qquad = \mathds{1}_{\Delta_i} (u) \mathds{1}_{\Delta_j} (v) \int_{u_k}^{u_{k+1}} \left( \int_{v_l}^{v_{l+1}} \left[ \psi \begin{pmatrix}
u_i & u_k \\
v_j & v_l
\end{pmatrix} - \psi \begin{pmatrix}
u & r_1 \\
v & r_2
\end{pmatrix} \right] K(\mathrm{d}r_2, v) \right) K(\mathrm{d}r_1, u).
\end{align*}
In addition, we can rewrite the integrand as
\begin{align*}
&\psi \begin{pmatrix}
u_i & u_k \\
v_j & v_l
\end{pmatrix} - \psi \begin{pmatrix}
u & r_1 \\
v & r_2
\end{pmatrix} \\
&\qquad= \psi \begin{pmatrix}
u & u_k \\
v & v_l
\end{pmatrix} + \psi \begin{pmatrix}
u_i & u \\
v_j & v
\end{pmatrix} + \psi \begin{pmatrix}
u & u_k \\
v_j & v
\end{pmatrix} + \psi \begin{pmatrix}
u_i & u \\
v & v_l
\end{pmatrix} 
- \psi \begin{pmatrix}
u & u_k \\
v & v_l
\end{pmatrix} - \psi \begin{pmatrix}
u_k & r_1 \\
v & r_2
\end{pmatrix} - \psi \begin{pmatrix}
u & u_k \\
v_l & r_2
\end{pmatrix} \\
&\qquad= \psi \begin{pmatrix}
u_i & u \\
v_j & v
\end{pmatrix} + \psi \begin{pmatrix}
u & u_k \\
v_j & v
\end{pmatrix} + \psi \begin{pmatrix}
u_i & u \\
v & v_l
\end{pmatrix} - \psi \begin{pmatrix}
u_k & r_1 \\
v & r_2
\end{pmatrix} - \psi \begin{pmatrix}
u & u_k \\
v_l & r_2
\end{pmatrix},  
\end{align*}
which is essentially taking the difference between the two rectangular increments over $[u_i, u_k) \times [v_j, v_l)$ and $[u, r_1) \times [v, r_2)$ by canceling out the rectangular increment over the common region $[u, u_k) \times [v, v_l)$. Now, using strong bi-H\"{o}lder continuity and the fact that $u \leq u_k \leq r_1$ and $v \leq v_l \leq r_2$, one can verify that
$\norm{\psi \begin{pmatrix}
u_i & u_k \\
v_j & v_l
\end{pmatrix} - \psi \begin{pmatrix}
u & r_1 \\
v & r_2
\end{pmatrix}}_E$ is bounded above by
\begin{align*}
C \left[ \abs{u - u_i}^{\lambda} \abs{v - v_j}^{\lambda} + \abs{r_1 - u}^{\lambda} \left( \abs{v- v_j}^{\lambda} + \abs{r_2 - v_l}^{\lambda} \right) + \left( \abs{u - u_i}^{\lambda} + \abs{r_1 - u_k}^{\lambda} \right) \abs{r_2 - v}^{\lambda} \right].
\end{align*}
We thus obtain the bound
\begin{align} \label{newEstimate3}
\begin{split}
&\norm{ \sum_{k \geq i+1, l \geq j+1} \int_{u_k}^{u_{k+1}} \int_{v_l}^{v_{l+1}} f(r_1, r_2) \wrt{r_2}\wrt{r_1} }_E \\
&\qquad \qquad \leq C \left( \xi_{i, 4}(u) \xi_{j, 4}(v) + \left[\xi_{i, 3}(u) + \xi_{i, 4}(u) \right] (T - v)^{\lambda - \alpha} + (T - u)^{\lambda - \alpha} \left[ \xi_{j, 3}(v) + \xi_{j, 4}(v)\right] \right)
\end{split}
\end{align}
when $(u, v) \in  [u_i, u_{i+1}) \times [v_j, v_{j+1})$. \par 
From \eqref{newEstimate1}, \eqref{newEstimate2} and \eqref{newEstimate3}, we see that $\norm{ B^K (\psi^{\pi} - \psi) (u, v) }_E $ is bounded above by (up to multiplication by a constant)
\begin{align} \label{A4estimate}
\begin{split}
&\left( \sum_i \mathds{1}_{\Delta_i} (u) \xi_{i, 4}(u) \right) \left( \sum_j \mathds{1}_{\Delta_j} (v) \xi_{j, 4} (v) \right) + \left( \sum_i \mathds{1}_{\Delta_i} (u) \left[\xi_{i, 2} (u) + \xi_{i, 3}(u) + \xi_{i, 4}(u) \right] \right) (T - v)^{\lambda - \alpha} \\
&\qquad+ (T - u)^{\lambda - \alpha} \left( \sum_j \mathds{1}_{\Delta_j} (v) \left[ \xi_{i, 2} (v) + \xi_{j, 3}(v) + \xi_{j, 4}(v)\right] \right),
\end{split}
\end{align}
and thus
\begin{align*}
\int_{[0, T]^2} \norm{ B^K(\psi^{\pi} - \psi) (u, v) }^2_E \wrt{u} \wrt{v} \rightarrow 0
\end{align*}
as $\norm{\pi} \rightarrow 0$.
	
(ii) To show the second result \eqref{result2}, we set $u = v = r$ and simply apply Cauchy-Schwarz to each term in \eqref{A1estimate}, \eqref{A2estimate} and \eqref{A4estimate} before using \eqref{xiConv}. 
\end{proof}

\end{document}